\documentclass[12 pt]{amsart}
\usepackage[colorinlistoftodos]{todonotes}

\usepackage{amsxtra}
\usepackage{geometry}
\usepackage{amssymb}
\usepackage{cite}
\usepackage{graphicx}
\usepackage[mathscr]{eucal}
\RequirePackage{stix}

\usepackage{thmtools} 
\usepackage{thm-restate}

\usepackage{enumerate}
\newenvironment{enum}{  
\begin{enumerate}[\upshape(\arabic{section}.\arabic{equation}a)] }  
{  \end{enumerate}   }
\newcommand{\itemref}[2] {{\upshape(\ref{#1}\ref{#2})}}

\theoremstyle{plain}
\newtheorem{Thm}{Theorem}

\newtheorem*{MThm-rev}{Theorem B [expanded statement]}

\newtheorem{Cor}[Thm]{Corollary}
\newtheorem{Prop}[Thm]{Proposition}
\newtheorem{Lem}[Thm]{Lemma}

\newtheorem{Def}[Thm]{Definition}

\theoremstyle{definition}
\newtheorem{Remark}[Thm]{Remark}
\newtheorem{Tech}[Thm]{Technical Remark}

\newtheorem{Ex}[Thm]{Example}

\renewcommand{\Im}{\operatorname{Im}}
\renewcommand{\bar}{\overline}
\renewcommand{\tilde}{\widetilde}
\newcommand{\intl}{\int\limits}
\newcommand{\C}{\mathbb C}
\newcommand{\R}{\mathbb R}
\newcommand{\CP}{\mathbb{CP}}
\newcommand{\dee}{\partial}
\newcommand{\deebar}{\overline\partial}
\newcommand{\vf}[1]{\dfrac{\dee}{\dee #1}}

\newcommand{\lma}{\langle\langle}
\newcommand{\rma}{\rangle\rangle}

\newcommand{\dual}{\mathscr D}
\newcommand{\dualx}{\mathscr D^\star}

\newcommand{\mapdef}[4]{ #1 &\to #2 \\ #3 &\mapsto #4 }
\newcommand{\bndry}{b}
\newcommand{\w}{\wedge}
\newcommand{\inv}{^{-1}}
\numberwithin{equation}{section}

\begin{document}
\title[Sums of CR functions from competing CR structures]
{Sums of CR functions from competing CR structures}
\author[David E. Barrett]{David E. Barrett}
\address{Department of Mathematics\\University of Michigan
\\Ann Arbor, MI  48109-1043  USA }
\email{barrett@umich.edu}

\author[Dusty E. Grundmeier]{Dusty E. Grundmeier}
\address{Department of Mathematics \\
Harvard University \\
 Cambridge, MA 02138-2901  USA }
\email{deg@math.harvard.edu}


\thanks{The first author was supported in part by NSF grants number DMS-1161735 and DMS-1500142.}

\date{\today}

\begin{abstract}  
In this paper we characterize sums of CR functions from  competing CR structures in two scenarios. 
In one scenario the structures are conjugate and we are adding to the theory of pluriharmonic boundary values.  In the second scenario the structures are related by projective duality considerations.  In both cases we provide explicit vector field-based characterizations for two-dimensional circular domains satisfying natural convexity conditions.
\end{abstract}

\maketitle

\section{Introduction}\label{S:Intro}

The Dirichlet problem for pluriharmonic functions is a natural problem in several complex variables with a long history going back at least to Amoroso \cite{Amo}, Severi \cite{Sev}, Wirtinger \cite{Wir}, and others. It was known early on that the problem is not solvable for general boundary data, so we may try to characterize the admissible boundary values with a system of tangential partial differential operators. This was first done for the ball by Bedford in \cite{Bed1}; see \S \ref{SS:ball} for details. More precisely, given a bounded domain $\Omega$ with smooth boundary $S$, we seek a system $\mathcal{L}$ of partial differential operators tangential to $S$ such that  a function $u\in \mathcal{C}^{\infty}(S,\C)$ satisfies $\mathcal{L}u=0$  if and only if there exists $U\in \mathcal{C}^{\infty}(\overline{\Omega})$ such that $U|_S = u$ and $\partial \bar{\partial} U =0$. The problem may also be considered locally.

While natural in its own right, this problem also arises in less direct fashion in many areas of 
complex analysis and geometry.  For instance, this problem plays a fundamental role in Graham's work on the Bergman Laplacian \cite{Gra}, Lee's work on pseudo-Einstein structures \cite{Lee}, and Case, Chanillo, and Yang's work on CR Paneitz operators (see \cite{CCY} and the references therein). From another point of view, the existence of non-trivial restrictions on pluriharmonic boundary values points to the need to look elsewhere (such as to the Monge-Amp\`ere equations studied in  \cite{BeTa}) for Dirichlet problems solvable for general boundary data.

The pluriharmonic boundary value problem is closely related to the problem of characterizing sums of CR functions from different, competing CR structures; indeed, when the competing CR structures are conjugate then these problems coincide (in simply-connected settings);  see Propositions \ref{T:glo-prob} and \ref{T:loc-prob} below. Another natural construction leading to competing CR structures arises from the study of projective duality (see \S \ref{S:proj} or \cite{Bar} for precise definitions). 

In each of these two scenarios, we precisely characterize sums of CR functions from the two competing CR structures in the setting of two-dimensional circular domains satisfying appropriate convexity conditions.
For conjugate structures we assume strong pseudoconvexity; our result appears as   Theorem \ref{T:cir-plh} below. In the projective duality scenario we assume strong convexity (the correct assumption without the circularity assumption would be strong $\C$-convexity, but these notions coincide in the circular case; see \S \ref{S:pdh}), and the main result appears as Theorem \ref{T:cir-proj} below (with an expanded version appearing later in \S \ref{S:cir-proj}).  Our techniques for these two related problems are interconnected to a surprising extent, and the reader will notice that the projective dual scenario actually turns out to have more structure and symmetry. 

\begin{restatable}{MThm}{conj}
\label{T:cir-plh}
Let $S\subset\C^2$ be a  strongly pseudoconvex circular hypersurface.  Then there exist
nowhere-vanishing tangential vector fields $X,Y$ on $S$ satisfying the following conditions.
\refstepcounter{equation}\label{N:XY-cond}
\begin{enum}
\item If $u$ is a smooth function on a relatively open subset of $S$, then $u$ is CR  if and only if $Xu=0$.
\label{I:aaa}
\item If $u$ is a smooth function on a relatively open subset of $S$, then $u$ is CR  if and only if
$Y\bar u=0$.\label{I:bbb}
\item If $S$ is compact, then a smooth function $u$ on $S$ is a pluriharmonic boundary value (in the sense of Proposition \ref{T:glo-prob} below) if and only if $XXYu=0$. \label{I:cir-plh-glo}
\item A smooth function $u$ on a relatively open subset of $S$ is a pluriharmonic boundary value (in the sense of Proposition \ref{T:loc-prob} below) if and only if $XXYu=0=\bar{XXY}u$.  \label{I:cir-plh-loc}
\end{enum}
\end{restatable}

\begin{restatable}{MThm}{proj}
\label{T:cir-proj}
Let $S\subset\C^2$ be a  strongly convex circular hypersurface.  Then there exist
nowhere-vanishing tangential vector fields $X,T$ on $S$ satisfying the following conditions.
\refstepcounter{equation}\label{N:XT-cond-pre}
\begin{enum}
\item  If $u$ is a smooth function on a relatively open subset of $S$, then $u$ is CR  if and only if $Xu=0$.\label{I:bb-pre}
\item If $u$ is a smooth function on a relatively open subset of $S$, then $u$ is dual-CR  if and only if $Tu=0$.\label{I:cc-pre}
\item If $S$ is compact, then a smooth function $u$ on $S$ is the sum of a CR function and a dual-CR function if and only if $XXTu=0$. \label{I:cir-proj-glo-pre}
\item If $S$ is simply-connected (but not necessarily compact), then a smooth function $u$ on  $S$ is  the sum of a CR function and a dual-CR function if and only if $XXTu=0=TTXu$. \label{I:cir-proj-loc-pre}
\end{enum}
\end{restatable}

This paper is organized as follows. In \S \ref{S:conj} we focus on the case of conjugate CR structures (the pluriharmonic case). In \S \ref{S:proj} we study the competing CR structures coming from projective duality. In \S \ref{S:pfs} we prove Theorem B, while Theorem A is proved in \S \ref{S:pfs2}.   The final section \S \ref{S:furth} includes a discussion of uniqueness issues.

\section{Conjugate structures}\label{S:conj}

\subsection{Results on the ball} \label{SS:ball}

Early work focused on the case of the ball $B^n$ in $\mathbb{C}^n$. In particular, Nirenberg observed that there is no second-order system of differential operators tangent to $S^3$ that exactly characterize pluriharmonic functions (see \S \ref{S:Nir} for more details). Third-order characterizations were developed by Bedford in the global case and Audibert in the local case (which requires  stronger conditions). To state these results, we define the tangential operators 
\begin{align}
L_{kl} = z_k \vf{\bar{z}_l}-z_l \vf{\bar{z}_k} && \bar{L_{kl}} = \bar{z}_k \vf{z_l}-\bar{z}_l \vf{z_k}
\end{align}
for $1\leq k,l \leq n$. 

\begin{Thm}[{[Bed1]}]\label{T:Bglobal}
Let $u$ be smooth on $S^{2n-1}$, then $$\bar{L_{kl}} \bar{L_{kl}} L_{kl} u=0$$ for $1\leq k,l \leq n$ if and only if $u$ extends to a pluriharmonic function on $B^n$. \end{Thm}

\begin{Thm}[{[Aud]}]\label{T:Alocal}  Let $S$ be a relatively open subset of $S^{2n-1}$,
and let $u$ be smooth on $S$. Then $$L_{jk} L_{lm} \bar{L_{rs}} u=0=\bar{L_{jk}} \bar{L_{lm}} L_{rs} u$$ for $1\leq j,k,l,m,r,s \leq n$ if and only if $u$ extends to a pluriharmonic function on a one-sided neighborhood of $S$. \end{Thm}

For a treatment of both of these results along with further details and examples, see \S 18.3 of \cite{Rud}.

\subsection{Other results}

Laville \cite{Lav1, Lav2} also gave a fourth order operator to solve the global problem. In \cite{BeFe} Bedford and Federbush solved the local problem in the more general setting where $b\Omega$ has non-zero Levi form at some point. Later in \cite{Bed2}, Bedford used the induced boundary complex $(\partial \bar{\partial})_b$ to solve the local problem in certain settings. In Lee's work \cite{Lee} on pseudo-Einstein structures, he gives a characterization for abstract CR manifolds using third order pseudohermitian covariant derivatives. Case, Chanillo, and Yang study when the kernel of the CR Paneitz operator characterizes CR-pluriharmonic functions (see \cite{CCY} and the references therein). 

\subsection{Relation to decomposition on the boundary}\label{SS:conj-bg}

Outside of the proof of Theorem \ref{T:pair} below, all forms, functions, and submanifolds will be assumed $\mathcal{C}^\infty$-smooth. 

\begin{Prop} \label{T:glo-prob} Let $S\subset\C^n$ be a compact connected and simply-connected real hypersurface, and let $\Omega$ be the bounded domain with boundary $S$. Then for $u\colon S\to \C$
the following conditions are equivalent:
\refstepcounter{equation}\label{N:plh-CR-glob}
\begin{enum}
\item $u$ extends to a (smooth) function $U$ on $\bar\Omega$ that is pluriharmonic on $\Omega$.  
\label{I:gloext}
\item $u$ is the sum  of a CR function and a conjugate-CR-function. \label{I:glodecomp}
\end{enum}
\end{Prop}

\begin{proof}
In the proof that \itemref{N:plh-CR-glob}{I:gloext} implies \itemref{N:plh-CR-glob}{I:glodecomp}, the CR term is the restriction to $S$ of an anti-derivative for $\dee U$ on a simply-connected one-sided neighborhood of $S$, and the conjugate-CR term is the restriction to $S$ of an anti-derivative for $\deebar U$ on a one-sided neighborhood of $S$ (adjusting one term by a constant as needed).

To see that \itemref{N:plh-CR-glob}{I:glodecomp} implies \itemref{N:plh-CR-glob}{I:gloext} we use the global CR extension result \cite[Thm.\,2.3.2]{Hor} to extend the terms to holomorphic and conjugate-holomorphic functions, respectively; $U$ is then the sum of the extensions.
\end{proof}

\begin{Prop} \label{T:loc-prob}
Let $S\subset\C^n$ be a simply-connected strongly pseudoconvex real hypersurface.
Then for $u\colon S\to \C$ the following conditions are equivalent:
\refstepcounter{equation}\label{N:plh-CR-loc}
\begin{enum}
\item \label{I:locext} there is an open subset $W$ of $\C^n$ with $S\subset\bndry W$ (with $W$ lying locally on the pseudoconvex side of $S$) so that $u$ extends to a (smooth) function $U$ on $W\cup S$ that is pluriharmonic on $W$. 
\item \label{I:locdecomp} $u$ is the sum  of a CR function and a conjugate-CR-function.
\end{enum}
\end{Prop}

\begin{proof} The proof follows the proof of Theorem \ref{T:glo-prob} above, replacing the global CR extension result by  the Hans Lewy local CR extension result as stated in \cite[Sec.\,14.1, Thm.\,1]{Bog}.  
\end{proof}


\section{Projective dual structures}\label{S:proj}

\subsection{Projective dual hypersurfaces}  \label{S:pdh}

Let $S\subset\C^n$ be an oriented real hypersurface with defining function $\rho$.  $S$ is said to be {\em strongly $\C$-convex} if $S$ locally equivalent via a projective transformation (that is, via an automorphism of projective space) to a strongly convex hypersurface; this condition is equivalent to either of the following two equivalent conditions:
\refstepcounter{equation}\label{N:strCconv-cond}
\begin{enum}
\item the second fundamental form for $S$ is positive definite on the maximal complex subspace $H_zS$ of each $T_zS$;
\medskip
\item the complex tangent (affine) hyperplanes for $S$ lie to one side (the ``concave side") of $S$ near the point of tangency with minimal order of contact.
\end{enum}

\begin{Thm}\label{T:gl-Cconv}
When $S$ is compact and strongly $\C$-convex the complex tangent hyperplanes for $S$ are in fact disjoint from the domain bounded by $S$.
\end{Thm}

\begin{proof}\,
\cite[\S2.5]{APS}.
\end{proof}

 We note that strongly $\C$-convex hypersurfaces are also strongly pseudoconvex.

A circular hypersurface (that is, a hypersurface invariant under rotations $z\mapsto e^{i\theta}z$) 
is strongly $\C$-convex if and only if it is strongly convex \cite[Prop.\,3.7]{Cer}.

The proper general context for the notion of strong $\C$-convexity is in the study of real hypersurfaces in complex projective space $\CP^n$ (see for example \cite{Bar} and \cite{APS}).

We specialize now to the two-dimensional case.

\begin{Lem}\label{L:w-def}  Let $S\subset\C^2$ be a compact strongly $\C$-convex hypersurface enclosing the origin.
Then there is a uniquely-determined map 
\begin{align*}
\mapdef{\dual\colon S}{\C^2\setminus\{0\}}{z}{w(z)=(w_1(z),w_2(z))}
\end{align*}
satisfying
\refstepcounter{equation}\label{N:w-def}
\begin{enum}
\item $z_1w_1+z_2w_2=1$ on $S$; \label{I:key-rel}
\item  the vector field 
\begin{equation*}\label{E:Ydef}
Y\eqdef w_2\vf{z_1}-w_1\vf{z_2}
\end{equation*}
 is tangent
to $S$.
Moreover, $Y$ annihilates conjugate-CR functions on any
relatively open subset of $S$. \label{I:tang(1,0)}
\end{enum}
\end{Lem}

\begin{proof} It is easy to check that \itemref{N:w-def}{I:key-rel} and \itemref{N:w-def}{I:tang(1,0)} force
\begin{align*}
w_1(z)&=\frac{\frac{\dee\rho}{\dee z_1}}{z_1\frac{\dee\rho}{\dee z_1}+z_2\frac{\dee\rho}{\dee z_2}}\\
w_2(z)&=\frac{\frac{\dee\rho}{\dee z_2}}{z_1\frac{\dee\rho}{\dee z_1}+z_2\frac{\dee\rho}{\dee z_2}}.
\end{align*}
establishing uniqueness.  Existence follows provided that the denominators do not vanish; but the vanishing of the denominators occurs precisely when the complex tangent line for $S$ at $z$ passes through the origin, and Theorem \ref{T:gl-Cconv} above guarantees that this does not occur under the given hypotheses.
\end{proof}

\begin{Remark}
It is clear from the proof that the conclusions of Lemma \ref{L:w-def} also hold under the assumption that $S$ is a (not necessarily compact) hypersurface satisfying
\begin{equation}\label{E:0spec}
\text{no complex tangent line for $S$ passes through the origin.}
\end{equation}
\end{Remark}

\begin{Remark}
Any tangential vector field annihilating  conjugate-CR functions will be a scalar multiple of $Y$.
\end{Remark}

\begin{Remark}
The complex  line tangent to $S$ at $z$ is given by
\begin{equation}\label{E:Ctan}
\{\zeta\in\C^2\colon w_1(z)\zeta_1+w_2(z)\zeta_2=1\}.
\end{equation}
\end{Remark}

\begin{Remark}\label{R:ann}
The maximal complex subspace $H_zS$ of each $T_zS$ is annihilated by the form $w_1\,dz_1+w_2\,dz_2$.
\end{Remark}

\begin{Prop}\label{P:loc-diff} For $S$ strongly $\C$-convex satisfying \eqref{E:0spec}, the map   
$\dual$ is a local diffeomorphism onto an immersed strongly $\C$-convex hypersurface $S^*$, with each maximal complex subspace $H_zS$ of $T_zS$ mapped (non-$\C$-linearly) by $\dual_z'$ onto the corresponding maximal complex subspace of $H_{w(z)}S^*$.  For  $S$ strongly $\C$-convex and compact, $S^*$ is an embedded strongly $\C$-convex hypersurface and $\dual$ is a diffeomorphism.
\end{Prop}

\begin{proof}\,
\cite[\S 6]{Bar}, \cite[\S2.5]{APS}.
\end{proof}

For $S$ strongly $\C$-convex satisfying \eqref{E:0spec} we may extend  $\dual$ to a smooth map on an open set in $\C^2$; the extended map $\dual^\star$ will be a local diffeomorphism in some neighborhood $U$ of $S$.  We may then define  vector fields $\vf{w_1}, \vf{w_2}, \vf{\bar{w_1}}, \vf{\bar{w_2}}$ on $U$ by applying $\left(\left(\dualx\right)\inv\right)'$ to the corresponding vector fields on $\dualx(U)$; these newly-defined vector fields will depend on the choice of the extension 
$\dualx$.

\begin{Lem}\label{L:Vdef}
The non-vanishing vector field
\begin{equation*}
V\eqdef z_2\vf{w_1}-z_1\vf{w_2}\label{E:V-def}
\end{equation*}
is tangent to $S$ and is independent of the choice of the extension $\dualx$.
\end{Lem}

\begin{proof} From \itemref{N:w-def}{I:key-rel} we have
\begin{align*}
0 &= d(z_1w_1+z_2w_2)\\
&= z_1\,dw_1+z_2\,dw_2 + w_1\,dz_1 + w_2\,dz_2
\end{align*}
on $T_zS$.

From Remark \ref{R:ann} we deduce that the null space in $T_z\C^2$  of $z_1\,dw_1+z_2\,dw_2$ is precisely the maximal complex subspace $H_z S$ of $T_z S$ 
(and moreover the null space in $\left(T_z\C^2\right)\otimes\C$ of of $z_1\,dw_1+z_2\,dw_2$ is precisely $\left(H_z S\right)\otimes\C$).  
If we apply  $z_1\,dw_1+z_2\,dw_2$  to $V$ we obtain
\begin{equation*}
z_1\cdot Vw_1+ z_2\cdot Vw_2 = z_1 \cdot z_2 - z_2 \cdot z_1=0
\end{equation*}
showing that $V$ takes values in $\left(H_z S\right)\otimes\C$ and is thus tangential.

If  an alternate tangential vector field $\tilde V$ is constructed with the use of an alternate extension $\widetilde{\dualx}$ of $\dual$, then 
\begin{align*}
\tilde Vw_j&=\pm z_{3-j}=Vw_j\\
\tilde V\bar w_j &= 0 = V\bar w_j
\end{align*}
along $S$, so $\tilde V = V$ along $S$.
\end{proof}

\begin{Def}\label{D:dCR}
A function $u$ on a relatively open subset of $S$ will be called {\em dual-CR} if $\bar Vu=0$.
\end{Def}

\begin{Ex}
If $S$ is the unit sphere in $\C^2$, then $w(z)=\bar z$ and the set of dual-CR functions on $S$ coincides with the set of conjugate-CR functions on $S$.
\end{Ex}

The set of dual-CR functions will only rarely coincide with the set of conjugate-CR functions as we see from the following two related results.

\begin{Thm} If $S$ is a compact strongly $\C$-convex hypersurface in $\C^2$, then the 
 set of dual-CR functions on $S$ will coincide with the the set of conjugate-CR functions on $S$ if and only if $S$ is a complex-affine image of the unit sphere.
\end{Thm}

\begin{Thm} If $S$ is a  strongly $\C$-convex hypersurface in $\C^2$, then the 
 set of dual-CR functions on $S$ will coincide with the the set of conjugate-CR functions on $S$ if and only if $S$ is locally the image of a relatively open subset of the unit sphere by a projective transformation.
\end{Thm}

For proofs of these results see \cite{Jen}, \cite{DeTr}, and \cite{Bol}.

\begin{Remark}
The constructions of the vector fields $Y$ and $V$ transform naturally under complex-affine mapping of $S$.  The construction of the dual-CR structure transforms naturally under projective transformation of $S$. (See for example \cite[\S 6]{Bar}.)
\end{Remark}

\begin{Lem}\label{L:YV-rel}
Relations of the form
\begin{align*}
V&= \chi Y + \sigma \bar Y\\
Y&= \kappa V + \xi \bar V
\end{align*}
hold along $S$ with $\sigma$ and $\xi$ nowhere vanishing.
\end{Lem}

\begin{proof}
This follows from the following facts:
\begin{itemize}
\item $V, \bar V, Y$ and $\bar Y$ all take values in the two-dimensional space $\left(H_z S\right)\otimes\C$;
\item $V$ and $\bar V$ are $\C$-linearly independent, as are $Y$ and $\bar Y$;
\item the non-$\C$-linearity of the map $\dual'_z\colon \left(H_z S\right)\otimes\C\to \left(H_z S^*\right)\otimes\C$ (see Proposition \ref{P:loc-diff}). 
\end{itemize}
 
\end{proof}

\begin{Lem}
If $f_1, f_2$ are CR functions and $g_1, g_2$ are dual-CR functions on a connected relatively open subset $W$ of $S$ with $f_1+g_1=f_2+g_2$, then $g_2-g_1=f_1-f_2$ is constant.
\end{Lem}

\begin{proof}
From Lemma \ref{L:YV-rel} we deduce that the directional derivatives of $g_2-g_1=f_1-f_2$ vanish in every direction belonging to the maximal complex subspace of $TS$.  Applying one Lie bracket we find that in fact all directional derivatives along $S$  of $g_2-g_1=f_1-f_2$ vanish.
\end{proof}

\begin{Cor}\label{C:sc}
If $W$ is a simply-connected relatively open subset of $S$ and $u$ is a function on $W$ that is locally decomposable as the sum of a CR function and a dual-CR function, then $u$ is decomposable on all of $W$ as the sum of a CR function and a dual-CR function.
\end{Cor}

\subsection{Circular hypersurfaces in $\C^2$} \label{S:cir-proj}

We begin the section with an expanded restatement of  the main theorem in the projective setting.

\begin{MThm-rev}
\label{T:cir-proj-rev}
Let $S\subset\C^2$ be a  strongly ($\C$-)convex circular hypersurface.  Then there exist
scalar functions  $\phi$ and $\psi$ on $S$ so that the vector fields
\begin{subequations}\label{E:aa}
\begin{align}
X&= V+\phi \bar{V}\\
T&=  Y+\psi \bar{Y}
\end{align}
\end{subequations}
satisfy the following conditions.
\refstepcounter{equation}\label{N:XT-cond}
\begin{enum}
\item  If $u$ is a smooth function on a relatively open subset of $S$, then $u$ is CR  if and only if $Xu=0$; equivalently, $X$ is a non-vanishing scalar multiple $\alpha \bar Y$ of $\bar Y$.\label{I:bb}
\item If $u$ is a smooth function on a relatively open subset of $S$, then $u$ is dual-CR  if and only if $Tu=0$; equivalently, $T$ is a non-vanishing scalar multiple $\beta \bar V$ of  $\bar V$.\label{I:cc}
\item If $S$ is compact, then a smooth function $u$ on $S$ is the sum of a CR function and a dual-CR function if and only if $XXTu=0$. \label{I:cir-proj-glo}
\item If $S$ is simply-connected (but not necessarily compact), then a smooth function $u$ on  $S$ is  the sum of a CR function and a dual-CR function if and only if $XXTu=0=TTXu$. \label{I:cir-proj-loc}
\end{enum}
\end{MThm-rev}

As we shall see the vector field $X$ in Theorem \ref{T:cir-proj} will also work as the vector field $X$ in Theorem \ref{T:cir-plh}.

\begin{Ex} (Compare \cite{Aud}.)
The function $\frac{z_1}{w_2}$ satisfies $XXT\frac{z_1}{w_2}=0$ but is not globally defined.  Since $TTX\frac{z_1}{w_2}=2\ne 0$ this function is not locally the sum of a CR function and a dual-CR function.
\end{Ex}

Conditions \eqref{E:aa}, \itemref{N:XT-cond}{I:bb} and \itemref{N:XT-cond}{I:cc} uniquely determine $X$ and $T$.  See \S \ref{S:uniq} for some discussion of what can happen without condition \eqref{E:aa}.

\section{Proof of Theorem \ref{T:cir-proj}}\label{S:pfs}

To prove Theorem \ref{T:cir-proj} we start by consulting Lemma \ref{L:YV-rel}
and note that \eqref{E:aa}, \itemref{N:XT-cond}{I:bb} and \itemref{N:XT-cond}{I:cc} will hold if we set 
\begin{align*}
\alpha&=1/\bar\xi,&
 \beta&=1/\bar\sigma,\\
  \phi&=\bar\kappa/\bar\xi,&
   \psi&=\bar\chi/\bar\sigma;
\end{align*}
it remains to check \itemref{N:XT-cond}{I:cir-proj-glo} and \itemref{N:XT-cond}{I:cir-proj-loc}.

We note for future reference and the reader's convenience that
\begin{align}
Xw_1&=z_2&
Xw_2&=-z_1\notag\\
\bar Yw_1&=\bar\xi z_2&
\bar Yw_2&=-\bar\xi z_1\notag\\
X\bar{w}_1 &= \phi \bar{z}_2 & X \bar{w}_2&=-\phi \bar{z}_1\notag\\
Xz_1&=\bar Yz_1=0&
Xz_2&=\bar Yz_2=0\notag\\
X\bar z_1 &= \alpha \bar w_2&
X\bar z_2 &= -\alpha \bar w_1\notag\\
Tz_1&=w_2&
Tz_2 &= -w_1\label{E:diff-rules}\\
\bar V z_1 &= \bar\sigma w_2&
\bar V z_2 &= -\bar\sigma w_1\notag\\
T \bar{z}_1 &= \psi \bar{w}_2 & T\bar{z}_2 &= -\psi \bar{w}_1 \notag\\
Tw_1&=\bar Vw_1=0&
Tw_2&=\bar Vw_2=0\notag\\
T\bar w_1 &= \beta z_2&
T\bar w_2 &= -\beta z_1.\notag
\end{align}

\begin{Lem}\label{L:YVbrack}
\begin{align*}
[Y,\bar Y] &=  \bar\xi  \left(z_1 \vf{ z_1}+  z_2 \vf{ z_2}\right)-\xi \left(\bar z_1 \vf{\bar z_1} +  \bar z_2 \vf{\bar z_2}\right)
  \\
[V,\bar V] &= \bar\sigma  \left(w_1 \vf{ w_1}+  w_2 \vf{ w_2}\right)
-\sigma \left(\bar w_1 \vf{\bar w_1} +  \bar w_2 \vf{\bar w_2}\right).
 \end{align*}
\end{Lem}

\begin{proof}
The first statement follows from
\begin{equation*}
[Y,\bar Y] = \left(Y\bar w_2\right)\vf{\bar z_1}-\left(Y\bar w_1\right)\vf{\bar z_2}
-\left(\bar Y w_2\right)\vf{ z_1}+\left(\bar Y w_1\right)\vf{ z_2}
\end{equation*}
along with \eqref{E:diff-rules}.  

The proof of the second statement is similar.
\end{proof}

We note that the assumption that $S$ is circular has not been used so far in this section.  We now bring it into play by introducing the real tangential vector field
\begin{equation*}
R  \eqdef i \left( z_1 \vf{z_1}+ z_2 \vf{z_2} - 
 \bar z_1 \vf{\bar z_1}- \bar z_2 \vf{\bar z_2}\right)
\end{equation*}
generating the rotations of $z\mapsto e^{i\theta}z$ of $S$.

\begin{Lem} The following hold.
\refstepcounter{equation}\label{N:rot-lem1}
\begin{enum}
\item $\bar\xi=\xi$ \label{I:xir}
\item $\bar\sigma=\sigma$ \label{I:sigmar}
\item $\bar\alpha=\alpha$ \label{I:alphar}
\item $\bar\beta=\beta$ \label{I:betar}
\item $R=-i \left( w_1 \vf{w_1}+ w_2 \vf{w_2} - 
 \bar w_1 \vf{\bar w_1}- \bar w_2 \vf{\bar w_2}\right)$ \label{I:wrot}
 \item $[Y,\bar Y]=-i\xi R$ \label{I:YbarY}
 \item $[V,\bar V]=i\sigma R$ \label{I:VbarV}
 \item $[X,Y]=iR-(Y\alpha)\bar Y$ \label{I:XY*}
\end{enum}
\end{Lem}

\begin{proof}
We start by considering the tangential vector field
\begin{equation*}
[Y,\bar Y]+i\xi R = (\bar \xi - \xi)\left( z_1 \vf{z_1}+ z_2 \vf{z_2}\right);
\end{equation*}
if \itemref{N:rot-lem1}{I:xir} fails, then $z_1 \vf{z_1}+ z_2 \vf{z_2}$ is a non-vanishing holomorphic tangential vector field on some non-empty relatively open subset of $S$, contradicting the strong pseudoconvexity of $S$.

To prove \itemref{N:rot-lem1}{I:wrot} we first note from Lemma \ref{L:w-def} that $w\left(e^{i\theta}z\right)=e^{-i\theta} w(z)$; differentiation with respect to $\theta$ yields \itemref{N:rot-lem1}{I:wrot}.

The proof of \itemref{N:rot-lem1}{I:xir} now may be adapted to prove \itemref{N:rot-lem1}{I:sigmar}.
\itemref{N:rot-lem1}{I:alphar} and \itemref{N:rot-lem1}{I:betar} follow immediately.

 Using Lemma \ref{L:YVbrack} in combination with \itemref{N:rot-lem1}{I:xir} and \itemref{N:rot-lem1}{I:sigmar} we obtain \itemref{N:rot-lem1}{I:YbarY} and \itemref{N:rot-lem1}{I:VbarV}.
 
 From \itemref{N:XT-cond}{I:bb} and \itemref{N:rot-lem1}{I:YbarY} we obtain \itemref{N:rot-lem1}{I:XY}.
\end{proof}

\begin{Lem}\label{L:XTbrack}
$[X,T]=iR$.
\end{Lem}

\begin{proof}
On the one hand, 
\begin{align*}
[X,T] &= [V+\phi\bar V,\beta\bar V]\\
&=\left( (V+\phi\bar V)\beta-\beta(\bar V\phi)
\right)\bar V + i \beta\sigma R\\
&= \left( (V+\phi\bar V)\beta-\beta(\bar V\phi)
\right)\bar V + i  R.
\end{align*}
On the other hand,
\begin{align*}
[X,T] &= [\alpha \bar Y,Y + \psi\bar Y]\\
&=\left( \alpha(\bar Y\psi)-(Y+\psi\bar Y)\alpha
\right)\bar Y + i \alpha \xi R\\
&=\left( \alpha(\bar Y\psi)-(Y+\psi\bar Y)\alpha
\right)\bar Y + i  R.
\end{align*}
Since $\bar V$ and $\bar Y$ are linearly independent, it follows that $[X,T]= i R$.
\end{proof}

\begin{Lem} The following hold.
\refstepcounter{equation}\label{N:Rbrack}
\begin{enum}
\item $[R,Y]=-2iY$ \label{I:RY}
\item $[R,\bar Y]=2i\bar Y$ \label{I:RbY}
\item $[R,V]=2iV$ \label{I:RV}
\item $[R,\bar V]=-2i\bar V$ \label{I:RbV}
\item $[R,X]=2iX$ \label{I:RX}
\item $[R,\bar X]=-2i\bar X$ \label{I:RbX}
\item $[R,T]=-2iT$ \label{I:RT}
\item $[R,\bar T]=2i\bar T$ \label{I:RbT}
\item $R\alpha=0$ \label{I:Ralpha}
\item $R\beta=0$ \label{I:Rbeta}
\end{enum}
\end{Lem}

\begin{proof}
\itemref{N:Rbrack}{I:RY}, \itemref{N:Rbrack}{I:RbY}, \itemref{N:Rbrack}{I:RV} and \itemref{N:Rbrack}{I:RbV} follow from direct calculation.

For \itemref{N:Rbrack}{I:RT}  first note that writing $T=\beta\bar V$ and using \itemref{N:Rbrack}{I:RbV} we see that $[R,T]$ is a scalar multiple of $T$.  Then writing
\begin{equation*}
[R,T] = [R,Y+\psi\bar Y]= -2iY + (\text{multiple of }\bar Y)
\end{equation*}
we conclude using \eqref{E:aa} that $[R,T]=-2iT$.  The proof of \itemref{N:Rbrack}{I:RX} is similar, and \itemref{N:Rbrack}{I:RbX} and \itemref{N:Rbrack}{I:RbT}  follow by conjugation.

Using \itemref{N:XT-cond}{I:bb} along with \itemref{N:Rbrack}{I:RbY} and \itemref{N:Rbrack}{I:RX} we obtain \itemref{N:Rbrack}{I:Ralpha}; \itemref{N:Rbrack}{I:Rbeta} is proved similarly.
\end{proof}

\begin{Lem}\label{L:XXker}
$XXf=0$ if and only if $f=f_1w_1+f_2w_2$ with $f_1, f_2$ CR.
\end{Lem}

\begin{proof}
From \itemref{N:XT-cond}{I:bb} and \eqref{E:diff-rules} it is clear that $XX\left(f_1w_1+f_2w_2\right)=0$ if $f_1$ and $ f_2$  are CR.

For the other direction, suppose that $XXf=0$. Then setting 
\begin{align*}
f_1&\eqdef z_1f+w_2Xf\\
f_2&\eqdef z_2f-w_1Xf.
\end{align*}
 it is clear that $f=f_1w_1+f_2w_2$; with the use of \itemref{N:XT-cond}{I:bb} and \eqref{E:diff-rules} it is also easy to check that $f_1$ and $f_2$ are CR.
\end{proof}

\begin{Lem}\label{L:proj-1eq}
Suppose that $XXTu=0$ so that by Lemma \ref{L:XXker} we may write $Tu=f_1w_1+f_2w_2$ with $f_1, f_2$ CR.  Then 
\begin{equation}\label{E:proj-2op}
TTXu = \frac{\dee f_1}{\dee z_1} + \frac{\dee f_2}{\dee z_2}.
\end{equation}
In particular, $TTXu$ is CR.
\end{Lem}

The non-tangential derivatives appearing in \eqref{E:proj-2op} may be interpreted using the Hans Lewy local CR extension result previously mentioned in the proof of Theorem \ref{T:loc-prob}, or else by rewriting them in terms of tangential derivatives (as in the last step of the proof below).

\begin{proof} We have
\begin{align*}
TTXu
&= TXTu + T[T,X]u \\
&= TX\left(f_1w_1+f_2w_2\right)-iTRu && \text{(Lemma \ref{L:XTbrack})}\\
&= T\left(f_1z_2-f_2z_1\right) - iRTu - i[T,R]u && \text{\itemref{N:XT-cond}{I:bb}, \eqref{E:diff-rules}  }\\
&= T\left(f_1z_2-f_2z_1\right) - iR\left(f_1w_1+f_2w_2\right) + 2Tu
&& \text{ \itemref{N:Rbrack}{I:RT}  }\\
&= \left(Tf_1\right)z_2 - f_1w_1 - \left(Tf_2\right)z_1 - f_2w_2 \\
&\qquad
-i\left(Rf_1\right)w_1 - f_2w_2 - i\left(Rf_2\right)w_2 - f_2 w_2 \\
&\qquad + 2\left(f_1w_1+f_2w_2\right)
&& \text{\eqref{E:diff-rules}, \itemref{N:rot-lem1}{I:wrot}  }\\
&= \left(z_2 T - i w_1 R\right) f_2 - \left( z_1 T + iw_2 R\right) f_2\\
&= \left(z_2 Y - i w_1 R\right) f_2 - \left( z_1 Y + iw_2 R\right) f_2\\
&= \frac{\dee f_1}{\dee z_1} + \frac{\dee f_2}{\dee z_2}.
\end{align*}
\end{proof}

\begin{Lem}\label{L:spec-2ord} The following hold.
\refstepcounter{equation}\label{N:spec-2ord}
\begin{enum}
\item The operator $XT$ maps CR functions to CR functions. \label{I:XT}
\item The operator $XY$ maps CR functions to CR functions. \label{I:XY}
\item The operator $TX$ maps dual-CR functions to dual-CR functions. \label{I:TX}
\item The operator $\bar{XY}$ maps conjugate-CR functions to conjugate-CR functions. \label{I:bar-XY}
\end{enum}
\end{Lem}

\begin{proof}
To prove \itemref{N:spec-2ord}{I:XT} and \itemref{N:spec-2ord}{I:XY} note that for $u$ CR we have $XTu=XYu=-z_1\frac{\dee u}{\dee z_1}-z_2\frac{\dee u}{\dee z_2}$  which is also CR.  The other proofs are similar.
\end{proof}

\begin{proof}[Proof of \itemref{N:XT-cond}{I:cir-proj-loc}] To get the required lower bound on the null spaces, it will suffice to show that $XXT$ and $TTX$ annihilate CR functions and dual-CR functions.  This follows from \itemref{N:XT-cond}{I:bb} and \itemref{N:XT-cond}{I:cc}  along with \itemref{N:spec-2ord}{I:XT} and \itemref{N:spec-2ord}{I:TX}.

For the other direction, if $XXTu=0=TTXu$, then from Lemma \ref{L:proj-1eq} we have a 
closed 1-form $\omega\eqdef f_2\,dz_1 - f_1\,dz_2$ on $S$ where $f_1$ and $f_2$ are CR functions satisfying $Tu=f_1w_1+f_2w_2$.  Since $S$ is simply-connected we may write $\omega=df$ with $f$ CR.  Then from \eqref{E:aa} we have
\begin{align*}
Tf &= Yf\\
&=w_2f_2+w_1f_1\\
&= Tu.
\end{align*}
Thus $u$ is the sum of the CR function $f$ and the dual-CR function $u-f$.
\end{proof}

To set up the proof of the global result \itemref{N:XT-cond}{I:cir-proj-glo} we introduce
the form
\begin{equation}\label{E:nu-def}
\nu\eqdef (z_2\,dz_1-z_1\,dz_2)\w dw_1\w dw_2
\end{equation}
and the $\C$-bilinear pairing
\begin{equation}\label{E:pair-def}
\lma \mu,\eta\rma \eqdef \intl_S \mu\eta\cdot\nu
\end{equation}
between functions on $S$ (but see Technical Remark \ref{R:tech} below).

\begin{Lem}\label{L:parts}
$\lma T\gamma, \eta \rma=-\lma \gamma, T\eta \rma$. 
\end{Lem}
\begin{proof}

\allowdisplaybreaks
\begin{align*}
\lma T\gamma, \eta \rma+\lma \gamma, T\eta \rma
&= \intl_S T(\gamma\eta)\cdot\nu\\
&= \intl_S \iota_T d(\gamma\eta)\cdot\nu\\
&= \intl_S d(\gamma\eta)\cdot \iota_T \nu\\
&= \intl_S d(\gamma\eta \cdot \iota_T \nu)- \intl_S \gamma\eta\cdot d(\iota_T \nu)\\
&= 0- \intl_S \gamma\eta\cdot d(\iota_T ((z_2\,dz_1-z_1\,dz_2)\w dw_1\w dw_2)\\
&= - \intl_S \gamma\eta\cdot d((z_2\cdot Tz_1-z_1\cdot Tz_2)\cdot dw_1\w dw_2)\\
&\qquad\qquad +\intl_S \gamma\eta\cdot d( (z_2\,dz_1-z_1\,dz_2)\cdot Tw_1\w dw_2)\\
&\qquad\qquad -\intl_S \gamma\eta\cdot d( (z_2\,dz_1-z_1\,dz_2)\w dw_1\cdot Tw_2)\\
&= - \intl_S \gamma\eta\cdot d((z_2w_2+z_1w_1) \,dw_1\w dw_2)+0-0\\
&= - \intl_S \gamma\eta\cdot d(dw_1\w dw_2)\\
&=0.
\end{align*}
Here we have quoted
\begin{itemize}
\item the definition \eqref{E:pair-def} of the pairing $\lma\cdot\rma$;
\item the Leibniz rule $\iota_T(\varphi_1\w\varphi_2)=(\iota_T\varphi_1)\w\varphi_2+(-1)^{\deg\varphi_1}\varphi_1\w(\iota_T\varphi_2)$ for the interior product $\iota_T$;
\item the fact that $S$ is integral for 4-forms;
\item Stokes' theorem;
\item the rules \eqref{E:diff-rules};
\item the relation \itemref{N:w-def}{I:key-rel}.
\end{itemize}
\end{proof}

\begin{Thm}\label{T:pair}
Let $\mu$ be a CR function on a compact strongly $\C$-convex hypersurface $S$.   Then $\mu=0$ if and  only if $\lma \mu, \eta \rma =0$ for all  dual-CR $\eta$ on $S$.
\end{Thm}

\begin{proof}\,
\cite[(4.3d) from Theorem 3]{Bar}. (Note also definition enclosing \cite[(4.2)]{Bar}.) 
\end{proof}

\begin{proof}[Proof of \itemref{N:XT-cond}{I:cir-proj-glo}]  Assume that 
$XXTu=0$.  Noting that $S$ is simply-connected, from  \itemref{N:XT-cond}{I:cir-proj-loc}  it suffices to prove that $TTXu=0$.  From Lemma \ref{L:proj-1eq} we know that  
$TTXu$ is CR. By Theorem \ref{T:pair} it will suffice to show that
\begin{equation*}
\lma TTXu, \eta \rma=0
\end{equation*}
for dual-CR $\eta$. But from Lemma \ref{L:parts} we have
\begin{align*}
\lma TTXu, \eta \rma&=- \lma TXu, T\eta \rma\\
&=0
\end{align*}
as required.
\end{proof}

\begin{Remark}\label{R:pd} From symmetry of formulas in Lemma \ref{L:w-def} and \ref{L:Vdef} we have that $X_{S^*}=\dual_* T_{S}, T_{S^*}=\dual_* X_{S}$ and $S^{**}=S$.  These facts serve to explain why the formulas throughout this section appear in dual pairs.
\end{Remark}

\begin{Tech}\label{R:tech}
In \cite{Bar} the pairing \eqref{E:pair-def} applies not to functions $\mu,\nu$ but rather to forms
$\mu(z)\,(dz_1\w dz_2)^{2/3}$, $\mu(w)\,(dw_1\w dw_2)^{2/3}$; the additional notation is important in \cite{Bar} for keeping track of invariance properties under projective transformation but is not needed here.

Note also that \eqref{E:pair-def} coincides (up to a constant) with the pairing (3.1.8) in \cite{APS} with $s=w_1\,dz_1+w_2\,dz_2$.
\end{Tech}

\section{Proof of Theorem \ref{T:cir-plh}}\label{S:pfs2}

For the reader's convenience we restate the main theorem in the conjugate setting. 

\conj*

It is not possible in general to have $Y=\bar X$.

\begin{Lem}\label{L:prlh-1eq}
Suppose that $XXYu=0$ so that by Lemma \ref{L:XXker} we may write $Yu=f_1w_1+f_2w_2$ with $f_1, f_2$ CR.  Then 
\begin{equation}\label{E:plh-2op}
\bar{XXY}u = \alpha\left(\frac{\dee f_1}{\dee z_1} + \frac{\dee f_2}{\dee z_2}\right).
\end{equation}
In particular, $\alpha\inv\bar{XXY}u$ is CR.
\end{Lem}

\begin{proof}
We have
\begin{align*}
\allowdisplaybreaks
\bar{XXY}u 
&= \bar{XYX}u
+ \bar{X} [\bar X,\bar Y] u\\
&=\bar{XY} \left( \alpha \left(f_1w_1+f_2 w_2\right) \right) 
+\bar{X}\left( -iR-(\bar Y\alpha)Y\right) u && \text{\itemref{N:XT-cond}{I:bb}, \itemref{N:rot-lem1}{I:alphar}, \itemref{N:rot-lem1}{I:XY}}\\
&=\bar X\left( \alpha\bar Y \left( f_1w_1+f_2 w_2 \right)\right)
-i\bar X R u \\
&=\bar X\left(f_1z_2-f_2 z_1 \right)
-iR \bar X  u
-i[\bar X, R] u && \text{\itemref{N:XT-cond}{I:bb}, \eqref{E:diff-rules}}\\
&= \bar X\left(f_1z_2-f_2 z_1 \right)
-iR \left( \alpha \left( f_1w_1+f_2 w_2 \right) \right)
+2 \bar X u && \text{\itemref{N:XT-cond}{I:bb}, \itemref{N:Rbrack}{I:RbX}}\\
&= (\bar X f_1)\cdot z_2-f_1\cdot\alpha w_1-(\bar X f_2) \cdot z_1
- f_2 \cdot\alpha w_2\\
&\qquad - i \alpha \left(  (Rf_1)\cdot w_1-f_1\cdot(iw_1)+(Rf_2)\cdot w_2-f_2\cdot(iw_2) \right)
\\
&\qquad
+ 2 \alpha \left(f_1w_1+f_2 w_2\right)  && \text{\eqref{E:diff-rules}, \itemref{N:Rbrack}{I:Ralpha}, \itemref{N:rot-lem1}{I:wrot}, \itemref{N:XT-cond}{I:bb}}\\
&= (\bar X f_1)\cdot z_2-(\bar X f_2) \cdot z_1
 - i \alpha \left(  (Rf_1)\cdot w_1+(Rf_2)\cdot w_2 \right)
  \\
  &= \alpha\left(z_2Y-iw_1R)f_1-(z_1Y+iw_2R)f_2\right)\\
&= \alpha \left( \frac{\dee f_1}{\dee z_1}+ \frac{\dee f_2}{\dee z_2}\right).\\
\end{align*}
\end{proof}

\begin{proof}[Proof of \itemref{N:XY-cond}{I:cir-plh-loc}] 
To get the required lower bound on the null spaces, it will suffice to show that $XXY$ and $\bar{XXY}$ annihilate CR functions and conjugate-CR functions.  This follows from \itemref{N:XY-cond}{I:aaa}  along with \itemref{N:spec-2ord}{I:XY} and \itemref{N:spec-2ord}{I:bar-XY}.

For the other direction, if $XXYu=0=\bar{XXY}u$, then from Lemma \ref{L:proj-1eq} we have a 
closed 1-form $\tilde\omega\eqdef f_2\,dz_1 - f_1\,dz_2$ on the open subset of $S$ where $f_1$ and $f_2$ are CR functions satisfying $Yu=f_1w_1+f_2w_2$.  Restricting our attention to a simply-connected subset, we may write $\omega=d f$ with $ f$ CR.  Then  we have
\begin{align*}
Y f
&=w_2f_2+w_1f_1\\
&= Yu.
\end{align*}
Thus $u$ is the sum of the CR function $f$ and the conjugate-CR function $u-f$. 

The general case follows by localization.
\end{proof}

\begin{Lem}\label{L:Ydiv}
$\operatorname{div} Y\eqdef \frac{\dee w_2}{\dee z_1} - \frac{\dee w_1}{\dee z_2}$ and $\operatorname{div} \bar Y\eqdef {\frac{\dee \bar w_2}{\dee \bar z_1}} - \bar{\frac{\dee \bar w_1}{\dee \bar z_2}}$ vanish on $S$.
\end{Lem}

\begin{proof}
Since $S$ is circular, any defining function $\rho$ for $S$ will satisfy 
$\Im \left( z_1\frac{\dee\rho}{\dee z_1}+z_2\frac{\dee\rho}{\dee z_2} \right)
=-\frac{R\rho}{2}  = 0$.  Adjusting our choice of defining function we may arrange that
$z_1\frac{\dee\rho}{\dee z_1}+z_2\frac{\dee\rho}{\dee z_2}\equiv 1$ in some neighborhood of 
$S$.  Then from the proof of Lemma \ref{L:w-def} we have $\frac{\dee w_2}{\dee z_1} - \frac{\dee w_1}{\dee z_2}=\frac{\dee^2 \rho}{\dee z_1 \dee z_2}-
\frac{\dee^2 \rho}{\dee z_2 \dee z_1}=0$.

The remaining statement follows by conjugation.
\end{proof}

\begin{Lem}\label{L:parts-plh}
$\displaystyle{\intl_S \left(X\gamma\right)\eta\,\frac{dS}{\alpha} = - \intl_S \gamma\left(X\eta\right)\,\frac{dS}{\alpha}}$
\end{Lem}
\begin{proof} 
\begin{align*}
\intl_S \left(X\gamma\right)\eta\,\frac{dS}{\alpha}
&= \intl_S \left(\bar Y\gamma\right)\eta\,dS && \text{\itemref{N:XT-cond}{I:bb}}\\
&= -\intl_S \gamma\left(\bar Y\eta\right)\,dS && \text{(Lemma \ref{L:Ydiv})}\\
&= - \intl_S \gamma\left(X\eta\right)\,\frac{dS}{\alpha} && \text{\itemref{N:XT-cond}{I:bb}}
\end{align*}
(The integration by parts above may be justified by applying the divergence theorem on a tubular neighborhood of $S$ and passing to a limit.)

 \end{proof}

\begin{proof}[Proof of \itemref{N:XY-cond}{I:cir-plh-glo}]
Assume that 
$XXYu=0$.  Noting that $S$ is simply-connected, from  \itemref{N:XY-cond}{I:cir-plh-loc}  it suffices to prove that $\bar{XXY}u=0$.  From Lemma \ref{L:proj-1eq} we know that  
$\alpha\inv\bar{XXY}u$ is CR.  The desired conclusion now follows from
\begin{align*}
\intl_S  \left|\bar{XXY}u\right|^2\,\frac{dS}{\alpha^2}
&=\intl_S \alpha\inv \bar{XXY}u\cdot {XXY}\bar u\,\frac{dS}{\alpha}\\
&= - \intl_S X\left( \alpha\inv\bar{XXY}u \right) \cdot {XY}\bar u\,\frac{dS}{\alpha}
&&\text{(Lemma \ref{L:parts-plh})}\\
&= - \intl_S 0 \cdot {XY}\bar u\,\frac{dS}{\alpha}
& & \text{(Lemma \ref{L:prlh-1eq})}\\
&=0.
\end{align*}

\end{proof}

\section{Further comments}\label{S:furth}

\subsection{Remarks on uniqueness}\label{S:uniq}

\begin{Prop}\label{T:plh-un}
Suppose that in the setting of Theorem \ref{T:cir-proj}  we have vector fields $\tilde X,\tilde T$ satisfying (suitably-modified) \itemref{N:XT-cond}{I:bb} and \itemref{N:XT-cond}{I:cc}. Then $\tilde X\tilde X\tilde T$ annihilates CR functions and dual-CR functions if and only if there are CR functions $f_1, f_2$ and $f_3$ so that $f_1w_1+f_2w_2$ and $f_3$ are non-vanishing and 
\begin{align*}
\tilde X&= f_3(f_1w_1+f_2w_2)^2 X\\
\tilde T&=  \frac{1}{f_1w_1+f_2w_2} T.
\end{align*}
\end{Prop}

\begin{proof}
From \itemref{N:XT-cond}{I:bb} and \itemref{N:XT-cond}{I:cc}  we have $\tilde X=\gamma X$, $\tilde T=\eta T$ with non-vanishing scalar functions $\gamma$ and $\eta$.

Suppose that $\tilde X\tilde X\tilde T$ annihilates CR functions and dual-CR functions.
By routine computation we have
\begin{equation*}
\tilde X\tilde X\tilde T
= \gamma^2\eta XXT + \gamma\Big( \left(2\gamma(X\eta)+\eta(X\gamma)\right) XT
+ (X(\gamma(X\eta))T\Big).
\end{equation*}

The operator $\left(2\gamma(X\eta)+\eta(X\gamma)\right) XT
+ (X(\gamma(X\eta))T$ must in particular annihilate CR functions.  
But if $f$ is CR, then using Lemma \ref{L:XTbrack} we have
\begin{equation*}
\Big( \left(2\gamma(X\eta)+\eta(X\gamma)\right) XT
+ (X(\gamma(X\eta))T\Big)f = \Big( i\left(2\gamma(X\eta)+\eta(X\gamma)\right) R
+ (X(\gamma(X\eta))T\Big)f
\end{equation*}
Since $R$ and $T$ are $\C$-linearly independent and $f$ is arbitrary it follows that we must have
\begin{align*}
X(\gamma\eta^2)=2\gamma(X\eta)+\eta(X\gamma) &= 0\\
X(\gamma(X\eta) &= 0.
\end{align*}

We set $f_3=\gamma\eta^2$ which is CR and non-vanishing.  Then the second equation above
yields
\begin{align*}
-f_3 \cdot XX(\eta\inv) 
&=   X\left(f_3 \,\eta^{-2}(X\eta)\right)\\
&=X(\gamma(X\eta)\\
&=0
\end{align*}
and hence $XX(\eta\inv)=0$.  From  Lemma \ref{L:XXker} we have
$\eta=\frac{1}{f_1w_1+f_2w_2}$ with $f_1$ and $f_2$ CR.  The result now follows.

The converse statement follows by reversing steps.
\end{proof}

\begin{Prop}
Suppose that in the setting of Theorem \ref{T:cir-plh}  we have vector fields $\tilde X,\tilde T$ satisfying (suitably-modified) \itemref{N:XY-cond}{I:aaa} and \itemref{N:XY-cond}{I:bbb}. Then $\tilde X\tilde X\tilde Y$ annihilates CR functions and conjugate-CR functions if and only if there are CR functions $f_1, f_2$ and $f_3$ so that $f_1w_1+f_2w_2$ and $f_3$ are non-vanishing and 
\begin{align*}
\tilde X&= f_3(f_1w_1+f_2w_2)^2 X\\
\tilde Y&=  \frac{1}{f_1w_1+f_2w_2} Y.
\end{align*}
\end{Prop}

The proof is similar to that of Proposition \ref{T:plh-un}, using \itemref{N:rot-lem1}{I:XY*} in place of Lemma \ref{L:XXker}.

\subsection{Nirenberg-type result}\label{S:Nir}

\begin{Prop}\label{P:Nir-plh}
 Given a point $p$ on a strongly pseudoconvex hypersurface $S\subset\C^2$, any 2-jet at $p$ of a $\C$-valued function on $S$ is the 2-jet of the restriction to $S$ of a pluriharmonic function on $\C^2$.
\end{Prop}

\begin{proof}
After performing a standard local biholomorphic change of coordinates we may reduce to the case where $p=0$ and $S$ is described near 0 by an equation of the form
\[y_2=z_1\bar z_1+O(\|(z_1,x_2)\|)^3.\]  The projection $(z_1,x_2+iy_2)\mapsto (z_1,x_2)$ induces a bijection between 2-jets at 0 along $S$ and 2-jets at 0 along $\C\times\R$.  It suffices now to note that the 2-jet 
\begin{equation*}
A+Bz_1+C\bar{z}_1+Dx_2+E z_1^2+ F \bar{z}_1^2 + G z_1\bar{z}_1 + H z_1 x_2 + I \bar{z}_1 x_2+Jx_2^2
\end{equation*}
is induced by the pluriharmonic polynomial
\begin{equation*}
A+Bz_1+C\bar{z}_1+\frac{D-iG}{2} z_2 + \frac{D+iG}{2} \bar{z}_2+ E z_1^2+ F \bar{z}_1^2  +
Hz_1z_2 + I \bar{z}_1\bar{z}_2 + J\bar{z}_2^2.
\end{equation*}

\end{proof}

\end{document}